\documentclass{amsart}

\usepackage{amsmath}
\usepackage{mathtools}
\usepackage{amssymb}
\usepackage{color}
\usepackage{amsthm}
\usepackage{tikz}
\usepackage{tikz-cd}
\usetikzlibrary{positioning,calc}
\usepackage{hyperref}
\usepackage{tensor}
\usepackage{relsize}

\theoremstyle{plain}
\newtheorem{theorem}{Theorem}[section]
\theoremstyle{definition}
\newtheorem{definition}{Definition}[section]
\theoremstyle{definition}
\newtheorem{remark}{Remark}[section]
\theoremstyle{definition}
\newtheorem{example}{Example}[section]
\theoremstyle{plain}
\newtheorem{prop}{Proposition}[section]
\theoremstyle{plain}
\newtheorem{lemm}{Lemma}[section]
\theoremstyle{plain}
\newtheorem{cor}{Corollary}[section]

\definecolor{Noir}{rgb}{0,0,0} 
\definecolor{Blanc}{rgb}{1,1,1} 
\definecolor{Gray}{rgb}{0.5,0.5,0.5} 
\definecolor{Rouge}{rgb}{0.8,0.1,0.1} 
\definecolor{DBleu}{RGB}{51,51,178} 
\definecolor{LBleu}{rgb}{0.85,0.85,1} 
\definecolor{Orange}{RGB}{255,140,0} 
 
\newcommand{\bcent}{\begin{center}} 
\newcommand{\ecent}{\end{center}} 
 
\newcommand{\benum}{\begin{enumerate}} 
\newcommand{\eenum}{\end{enumerate}} 
\newcommand{\bitem}{\begin{itemize}} 
\newcommand{\eitem}{\end{itemize}} 
 
\newcommand{\btab}{\begin{tabular}} 
\newcommand{\etab}{\end{tabular}} 
\newcommand{\beqn}{\begin{eqnarray}} 
\newcommand{\eeqn}{\end{eqnarray}} 
 
\newcommand{\bmath}{\begin{math}} 
\newcommand{\emath}{\end{math}} 
 
\newcommand{\noin}{\noindent} 

\providecommand{\F}[1]{\mathbb{#1}}

\newcommand{\ZZ}{\mathbb{Z}}

\newcommand{\NN}{\F{N}}

\newcommand{\CC}{\F{C}}

\newcommand{\PP}{\F P}

\newcommand{\plus}{\oplus}

\newcommand{\w}{\omega} 
 
\providecommand{\Cal}[1]{\mathcal{#1}} 
 
\newcommand{\CO}{\Cal O}

\newcommand{\calM}{\mathcal{M}}
\newcommand{\calO}{\mathcal{O}}

\newcommand{\calB}{\mathcal{B}}

\newcommand{\CM}{\Cal M}

\renewcommand{\epsilon}{\varepsilon}

\providecommand{\set}[1]{\left\{#1\right\}}

\newcommand\varleq{\mathbin{\vcenter{\hbox{%
  \oalign{\hfil$\scriptstyle<$\hfil\cr 
          \noalign{\kern-.3ex} 
          $\scriptscriptstyle({-})$\cr}%
}}}} 
 
\renewcommand\subsetneq{\mathbin{\vcenter{\hbox{%
  \oalign{\hfil$\scriptstyle\subset$\hfil\cr 
          \noalign{\kern-.3ex} 
          $\scriptscriptstyle({-})$\cr}%
}}}} 

\author{Steven Rayan}
\address{Department of Mathematics \& Statistics\\
  McLean Hall, University of Saskatchewan\\
  Saskatoon, SK, Canada  ~S7N 5E6}
\email{rayan@math.usask.ca, ejs383@mail.usask.ca}

\author{Evan Sundbo}

\title[Twisted Cyclic Quiver Varieties on Curves]{Twisted Cyclic Quiver Varieties on Curves}
\date{\today}

\begin{document}

\maketitle

\begin{abstract}
We study the algebraic geometry of twisted Higgs bundles of cyclic type along complex curves.  These objects, which generalize ordinary cyclic Higgs bundles, can be identified with representations of a cyclic quiver in a twisted category of coherent sheaves.  Referring to the Hitchin fibration, we produce a fibre-wise geometric description of the locus of such representations within the ambient twisted Higgs moduli space.  When the genus is $0$, we produce a concrete geometric identification of the moduli space as a vector bundle over an associated (twisted) $A$-type quiver variety; we count the number of points at which the cyclic moduli space intersects a Hitchin fibre; and we describe explicitly certain $\CC^\times$-flows into the nilpotent cone.  We also extend this description to moduli of certain twisted cyclic quivers whose rank vector has components larger than $1$.  We show that, for certain choices of underlying bundle, such moduli spaces decompose as a product of cyclic quiver varieties in which each node is a line bundle.
\end{abstract}

\tableofcontents

\section{Introduction}

Given a complex projective curve $X$ of genus $g\geq2$, a \emph{cyclic Higgs bundle} on $X$ is an $n$-tuple of holomorphic line bundles $L_i\to X$ with respective holomorphic maps $\phi_i\in H^0(X,L_i^{-1}\otimes L_{i+1}\otimes\omega_X)$, $\phi_n\in H^0(X,L_n^{-1}\otimes L_{1}\otimes\omega_X)$, where $\omega_X$ is the canonical line bundle of $X$.  The data of the $\phi_i$ maps can be arranged to form a matrix$$\Phi=\left(\begin{array}{ccccc} 0&\cdots & & & \phi_n\\\phi_1 &\ddots & & &\\ & \phi_2 & & &\\& & \ddots & &\\ & & & \phi_{r-1} &0\end{array}\right),$$which we regard as acting on sections of $E=\bigoplus L_i$ by multiplication.  This is a special case of data of the form $(E,\Phi)$, where $E\to X$ is a holomorphic vector bundle and $\Phi\in H^0(X,\mbox{End}(E)\otimes\omega_X)$. This data $(E,\Phi)$ is what is usually termed a \emph{Higgs bundle} on $X$.

Cyclic Higgs bundles generalize the elements of the Hitchin section in the moduli space of Higgs bundles, are closely related to the affine Toda equations \cite{DB:15}, and have been crucial to establishing some of the known cases \cite{Lab17} of Labourie's conjecture \cite{Lab06,Lof07,ColThoTou17}.  They also occupy a special position within nonabelian Hodge theory because they admit diagonal harmonic metrics solving Hitchin's equations.  At the same time, they capture essential information about the geometry of $\CC^\times$-flows on the moduli space of Higgs bundles.  Hitchin showed \cite{NJH:86} that the energy of $t\cdot(E,\Phi)=(E,t\Phi)$ grows with integral monotonicity as $t$ increases.  For cyclic Higgs bundles, this was improved to pointwise monotonicity in \cite{DaiLi17,DL:19}.  These observations are intertwined with the ``Hitchin WKB problem'' \cite{GaiMooNei13,MazSwoWeiWit14,Moc15,KNPS:15,MazSwoWeiWit16,DumFreKydMazMulNei16,Fre18,DumNei19} in the nonabelian Hodge theory of $X$, asking about the behaviour of harmonic metrics, harmonic maps, and their corresponding flat connections as $t$ goes to infinity in $(E,t\Phi)$.  Again, this problem has been resolved for cyclic Higgs bundles in the Hitchin section \cite{ColLi17}.

In this paper, we generalize the notion of cyclic Higgs bundles to incorporate higher-rank bundles $L_i$ and arbitrary twists.  In other words, we consider the data of an $n$-tuple of holomorphic vector bundles $U_i$ and maps $\phi_i\in H^0(X,U_i\otimes U_{i+1}\otimes L)$, $\phi_n\in H^0(X,U_n^{-1}\otimes U_1\otimes L)$, where $L$ is a fixed line bundle on $X$.  We refer to these as \emph{twisted cyclic Higgs bundles}.  The advantage here is that the degree of $L$ may be taken large enough so that such objects admit moduli even when $g=0,1$ (alternatively one may keep $\w_X$ but puncture $X$ sufficiently-many times and introduce a parabolic or irregular structure in $\Phi$ at these points, as introduced in \cite{BodYok96}).

We study the moduli problem for twisted cyclic Higgs bundles with the stability condition induced by viewing them as special instances of twisted Higgs bundles in general, i.e. Higgs bundles with $\Phi\in H^0(X,\mbox{End}(E)\otimes L)$.  In particular, we identify the moduli space with a space of quiver representations in an appropriately-defined $L$-twisted category of sheaves on $X$.  The underlying quiver is a cyclic one:
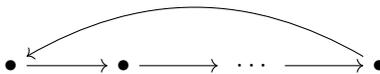
\begin{equation}\nonumber
\begin{tikzpicture}
\node (a) at (0,0){$\bullet$};
\node (b) at (1.5,0){$\bullet$};
\node(c) at (3.2,0){$\,\,\cdots\,\,$};
\node(d) at (4.9,0){$\bullet$};

\draw[->] (a) to (b);
\draw[->] (b) to (c);
\draw[->] (c) to (d);
\draw[->] (d) to[bend right,looseness=1] (a);
\end{tikzpicture}
\end{equation}

Such quivers are naturally related to $A$-type quivers.  From the vantage point of the $\CC^\times$-action on the moduli space of Higgs bundles, fixed points are twisted representations of $A$-type quivers and have been studied extensively, see for example \cite{PBG:95,AGS:06,SR:11,GHS:14,SM:16}.  When the rank and degree are coprime, the fixed points are also precisely where the moduli space of twisted cyclic Higgs bundles intersects the nilpotent cone of the ambient moduli space of general twisted Higgs bundles.  The stability condition here is the usual Mumford-Hitchin slope condition.

We refer to the vector defined by the ordered ranks of the $U_i$'s as the \emph{type} of a twisted cyclic Higgs bundle.  We study moduli spaces of representations of type $(1,\dots,1)$ cyclic quivers in Section \ref{1...1g}.  Our main result here is Theorem \ref{cyclic1...1} which gives a fibre-wise description of the moduli space for the groups $\mbox{SL}(n,\CC)$ and $\mbox{PGL}(n,\CC)$.  In Section \ref{1...1p1} we specialize to $g=0$, where we can give concrete descriptions of the moduli spaces of representations of $(1,\dots,1)$ cyclic quivers (Theorem \ref{thmcyclicp1}).  We describe two different ways to view this moduli space, one of which counts the number of points of intersection of our moduli space with the Hitchin system per fibre, and the other of which describes the cyclic quiver variety as a vector bundle over its corresponding $A$-type quiver variety.  These moduli spaces are also a special kind of (singular) weighted projective spaces.  We conclude in Section \ref{k1p1} with a further generalization to cyclic quivers on $\PP^1$  with other types of rank vectors.  We show in Theorem \ref{k1cyclicp1} that for certain choices of bundles, the moduli space decomposes as a product of type $(1,1)$ cyclic quiver varieties.\\

\noin\emph{Acknowledgements.}   We thank Peter Gothen and Qiongling Li for useful discussions.  The first named author was supported during this work by an NSERC Discovery Grant.  The second named author was supported by an NSERC Alexander Graham Bell Scholarship.  We thank the GEAR Network (NSF grants DMS 1107452, 1107263, 1107367 \emph{RNMS: Geometric Structures and Representation Varieties}) for supporting the workshop ``Geometry and Physics of Gauge Theories at Infinity'' (August 2018) co-organized by the first author, where formative steps in this work were taken. Parts of this work were completed during the Mathematisches Forschungsintitut Oberwolfach Workshop ``Geometry and Physics of Higgs Bundles'' and during the ``Geometry and Physics of Hitchin Systems'' Thematic Program at the Simons Center for Geometry and Physics (both in May 2019).  The first named author thanks the organizers of the former and both authors thank the organizers of the latter.\\

\section{Preliminaries}\label{prelim}

Fix a nonsingular, irreducible, projective complex curve or a smooth, compact, connected Riemann surface $X$ and a holomorphic line bundle $L\to X$.  We will use ``curve'' and ``Riemann surface'' interchangeably and will always retain the preceding hypotheses.  We will use $\mathcal O_X$ and $\omega_X$ to denote the structure and canonical line bundles, respectively.  We also denote by $\mbox{Jac}^0(X)[r]$ the (finite) group of $r$-torsion line bundles on $X$, equivalently the order-$r$ roots of unity in the divisor group of $X$.  When $r$ is understood, we use the shorthand $\Gamma$ to refer to this group.  We denote by $\calO_{\PP^n}(d)$ the unique isomorphism class of holomorphic line bundes of degree $d$ on the projective space $\PP^n$ and when $n=1$ we simply write $\calO(d)$.

\begin{definition}\label{defhiggs}
An \emph{$L$-twisted Higgs bundle on $X$} is a pair $(E,\Phi)$, where $E$ is a holomorphic vector bundle on $X$ and $\Phi$ is a vector bundle morphism $\Phi:E\to E\otimes L$.  The \emph{slope} of $(E,\Phi)$ is $\mu(E) = \frac{\text{deg}E}{\text{rk}E}$. An $L$-twisted Higgs bundle $(E,\Phi)$ is called \emph{stable} (resp. \emph{semistable}) if 
\begin{equation}\nonumber
\mu(F) < (\leq) \mbox{ }\mu(E)
\end{equation}
for all proper subbundles $F$ of $E$ which satisfy $\Phi(F)\subseteq F\otimes L$ (this condition is known as \emph{$\Phi$-invariance}).  We will sometimes denote $\mu(E)$ by $\mu_{tot}$.  The data of $\Phi$ is commonly referred to as a \emph{Higgs field} for $E$.
\end{definition}

\begin{definition} We say that two (twisted) Higgs bundles $(E,\Phi)$ and $(E',\Phi')$ on $X$ are \emph{equivalent} if there exists a vector bundle isomrophism $\Psi:E\stackrel{\cong}{\to}E'$ for which $\Phi = \Psi\Phi'\Psi^{-1}$. 
\end{definition}

We denote the moduli space of equivalence classes of (semi)stable $L$-twisted Higgs bundles on $X$ of rank $r$ and degree $d$  by $\mathcal{M}_{X,L}(r,d)$.   When $L=\omega_X$ is the canonical line bundle of the curve (i.e. the case of ordinary Higgs bundles), the moduli space admits two fibrations: a surjective, proper fibration onto an affine base and a vector bundle fibration that is only rationally defined, given by the cotangent bundle to the moduli space of (semi)stable vector bundles without Higgs fields.  For general $L$ the latter fibration is unavailable, while the former, known as the \emph{Hitchin fibration}, is always defined.  For $\mathcal{M}_{X,L}(r,d)$, this is given by the map
\begin{equation}\begin{split}\nonumber
h:\mbox{ }\mathcal{M}&_{X,L}(r,d) \to \mathcal{B}_r = \bigoplus_{i=1}^rH^0(X,L^{\otimes i})
\\
&(E,\Phi)\mapsto \text{char}_\lambda\Phi
\end{split}\end{equation}
 
 The projection $h$ is the \emph{Hitchin map} and the base $\calB_r=\bigoplus_{i=1}^rH^0(X,L^{\otimes i})$ is the \emph{Hitchin base}.  This fibration is well-known in the case that $g \geq 2$ and $L=\omega_X$ where it endows $\mathcal{M}_{X,L}(r,d)$ with the structure of a hyperkaehler completely integrable system, but exists for any $g$ and $L$.
 
 Technically, Definition \ref{defhiggs} defines $L$-twisted $\mbox{GL}(r,\CC)$-Higgs bundles.  For the groups $\mbox{SL}(r,\CC)$ and $\mbox{PGL}(r,\CC)$, we have respectively:
 
 \begin{definition}
 An \emph{$L$-twisted $\mbox{SL}(r,\CC)$}-Higgs bundle with determinant $P$ on $X$ is a pair $(E,\Phi)$ consisting of a rank $r$ holomorphic vector bundle $E\to X$ with $\mbox{det}(E)=\bigwedge^r(E)=P$ and a holomorphic map $\Phi:E\to E\otimes L$ with $\mbox{tr}(\Phi)=0$.  We denote the corresponding moduli space by $\CM_{X,L}(r,P)$.
 \end{definition}

 \begin{definition}
An \emph{$L$-twisted $\mbox{PGL}(r,\CC)$-Higgs bundle} on $X$ is an equivalence class $[(E,\Phi)]$ of $L$-twisted $\mbox{SL}(r,\CC)$-Higgs bundles where $(E,\Phi)$ and $(E',\Phi')$ are equivalent if $E\cong E'\otimes M$ for some line bundle $M$ over $X$ and $\Phi=\Phi'\otimes 1_M$.  We denote the corresponding moduli space by $\CM^{\Gamma}_{X,L}(r,P)$.
 \end{definition}

In the second definition, note that the condition $E\cong E'\otimes M$ forces $M^r=\CO_X$, and hence $M\in\Gamma$.  It follows that an alternative description of $\CM^\Gamma_{X,L}(r,P)$ is as the quotient $\mathcal M_{X,L}(r,P)/\Gamma$, where $\Gamma$ acts by tensor product.

When the underlying curve is $X=\PP^1$, the notions of $\mbox{GL}(r,\CC)$-, $\mbox{SL}(r,\CC)$-, and $\mbox{PGL}(r,\CC)$-Higgs bundles coincide and so we will not make any distinction between them.\footnote{On $\PP^1$, fixing $\deg E=d$ forces $\det E\cong\mathcal O(d)$, and so up to isomorphism there is no freedom in the choice of $P$.  Also, $\Gamma=\set{\mathcal O_X}$ is trivial.  Hence, from the point of view of $E$, there is no difference between a $\mbox{GL}$, an $\mbox{SL}$, and a $\mbox{PGL}$ bundle.  Strictly speaking, $\Phi$ must be trace-free in the $\mbox{SL}$ and $\mbox{PGL}$ case, while the $\mbox{GL}$ case has no such restriction.  This will make no impact on the analysis for $g=0$, however, given that twisted cyclic Higgs bundles will always naturally have trace $0$.}

 \begin{definition}
An \emph{$L$-twisted cyclic Higgs bundle} on $X$ is an $L$-twisted Higgs bundle $(E,\Phi)$ on $X$ of the form
\begin{equation}\nonumber
E=U_1\oplus\dots\oplus U_n\mbox{,  }\quad
\Phi =  \left(\begin{matrix}
 0&\cdots&&\phi_n\\
\phi_1&\ddots&& \\
&\ddots&&\\
0&&\phi_{n-1}&0\\
\end{matrix}\right)
\end{equation}
where $U_i$ are holomorphic line bundles on $X$ and $\phi_i:U_i\to U_{i+1}\otimes L$.  Note that the subscript is counted modulo $n$.
\end{definition}

This special class of Higgs bundles will be our main focus.  Now let us shift gears to quiver representations.  A quiver is a finite directed graph, and we will be considering representations in the category $\mbox{Bun}(X,L)$ whose objects are holomorphic vector bundles on $X$ and whose morphisms are maps between them graded by exterior powers of $L$: that is, $\mbox{Hom}^k(U,V)=\mbox{Hom}(U,V\otimes\Lambda^k L)$. A representation of $Q$ in this category amounts to a choice of a vector bundle $U_i$ for each node (with rank and degree subject to some labelling of $Q$) and of a map $\phi_{ij}:U_j\to U_{i}\otimes L$ to each arrow.  Such representations are sometimes called \emph{quiver bundles}.  It is clear that these representations can be interpreted as $L$-twisted Higgs bundles, by taking the direct sum of the $U_i$'s and by forming the matrix $\Phi=(\phi_{ij})$.  Using the notions of equivalence and stability induced by this correspondence, we can study the moduli space of stable representations of $Q$ in $\mbox{Bun}(X,L)$, which we denote $\mathcal{M}_{X,L}(Q)$.

We can also define representations whose corresponding Higgs bundle is an $\mbox{SL}(r,\CC)$- or $\mbox{PGL}(r,\CC)$-Higgs bundle.  We denote these respective moduli spaces of stable objects by $\mathcal{M}_{X,L}(Q,P)$ and $\mathcal{M}^\Gamma_{X,L}(Q,P)$.

\begin{definition}
A \emph{cyclic quiver} is a quiver of the form
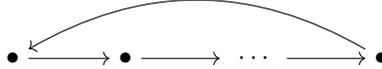
\begin{equation}\nonumber
\begin{tikzpicture}
\node (a) at (0,0){$\bullet$};
\node (b) at (1.5,0){$\bullet$};
\node(c) at (3.2,0){$\,\,\cdots\,\,$};
\node(d) at (4.9,0){$\bullet$};

\draw[->] (a) to (b);
\draw[->] (b) to (c);
\draw[->] (c) to (d);
\draw[->] (d) to[bend right,looseness=1] (a);
\end{tikzpicture}
\end{equation}
\end{definition}

It is clear that one can use cyclic quivers to study cyclic Higgs bundles (and vice versa).  Such quivers can be viewed as slight modifications of $A$-type quivers:
\begin{equation}\nonumber
\begin{tikzpicture}
\node (a) at (0,0){$\bullet$};
\node (b) at (1.5,0){$\bullet$};
\node(c) at (3.2,0){$\,\,\cdots\,\,$};
\node(d) at (4.9,0){$\bullet$};

\draw[->] (a) to (b);
\draw[->] (b) to (c);
\draw[->] (c) to (d);
\end{tikzpicture}
\end{equation} Representations of $A$-type quivers in $\text{Bun}(X,L)$ are often called \emph{holomorphic chains}, and moduli spaces of holomorphic chains are exactly the fixed points of the algebraic $\CC^\times$-action on $\calM_{X,L}(r,d)$.

\section{Examples: Hitchin section and analogues}

For the moment, let $X$ be a curve of genus $g\geq 2$ and consider specifically $\omega_X$-twisted $\mbox{SL}(2,\CC)$-Higgs bundles of the form
\begin{equation}\nonumber
E=\omega_X^{\frac{1}{2}}\oplus\omega_X^{-\frac{1}{2}}\mbox{,  }\quad
\Phi =  \left(\begin{matrix}
 0&q\\
1&0 \\
\end{matrix}\right)
\end{equation}
where $\omega_X^{\frac{1}{2}}$ is a choice of holomorphic square root of $\omega_X$ and $q:\omega_X^{-\frac{1}{2}} \to\omega_X^{\frac{1}{2}}\otimes\omega_X$.  That is to say that $q$ lies in $H^0(X,\omega_X^2)$, the space of holomorphic quadratic differentials on $X$.  In the $\mbox{SL}(2,\CC)$ moduli space, this is exactly the Hitchin base $\calB_2$.  All Higgs bundles of this form are stable, and so we have a map $\iota:\calB_2\to\calM_{X,\omega_X}(2,\mathcal O_X)$.  The image of this map is the \emph{Hitchin section}.  Alternatively, this a component of the moduli space of stable representations $\mathcal{M}_{X,\omega_X}(Q,\mathcal O_X)$ for the quiver
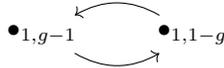
\begin{equation}\nonumber
\begin{tikzpicture}
\node (a) at (0,0){$\bullet_{1,g-1}$};
\node (b) at (2,0){$\bullet_{1,1-g}$};

\draw[->] (a) to [bend right](b);
\draw[->] (b) to[bend right] (a);
\end{tikzpicture}
\end{equation}

On $X=\PP^1$ there is an analogue of the Hitchin section in $\calM_{\PP^1,\calO(2)}(2,\calO)$, which we can motivate by studying representations of the cyclic quiver
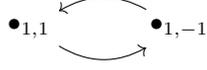
\begin{equation}\nonumber
\begin{tikzpicture}
\node (a) at (0,0){$\bullet_{1,1}$};
\node (b) at (2,0){$\bullet_{1,-1}$};

\draw[->] (a) to [bend right](b);
\draw[->] (b) to[bend right] (a);
\end{tikzpicture}
\end{equation}
which amounts to looking at the family of $(E,\Phi)$ of the form
\begin{equation}\nonumber
E=\calO(1)\oplus\calO(-1)\mbox{,  }\quad
\Phi =  \left(\begin{matrix}
 0&q\\
1&0 \\
\end{matrix}\right)
\end{equation}
with $q\in H^0(\PP^1,\calO(2)^{2})= H^0(\PP^1,\calO(4))$, i.e. the space of holomorphic quadratic vector fields on the sphere.  We have again formed a section.  However, in higher ranks (or with different labellings or twisting line bundles) the moduli space of stable representations of a cyclic quiver is not, in general, a section. 

\begin{remark}
There exists a different generalization of the Hitchin section which is in fact a section of $\calM_{X,\omega_X}(r,\calO_X)$ (cf. \cite{BGG:12,C:17,ABCGGO:18, Dum18}, for instance).  These sections can be viewed as representations of (appropriately labelled) quivers
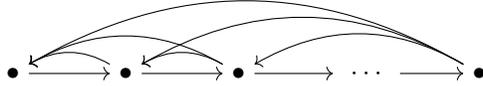
\begin{equation}\nonumber
\begin{tikzpicture}
\node (a) at (0,0){$\bullet$};
\node (b) at (1.5,0){$\bullet$};
\node (c) at (3,0){$\bullet$};
\node(d) at (4.7,0){$\,\,\cdots\,\,$};
\node(e) at (6.2,0){$\bullet$};

\draw[->] (a) to (b);
\draw[->] (b) to (c);
\draw[->] (c) to (d);
\draw[->] (d) to (e);
\draw[->] (e) to[bend right,looseness=1] (a);
\draw[->] (e) to[bend right,looseness=1] (b);
\draw[->] (e) to[bend right,looseness=1] (c);
\draw[->] (b) to[bend right,looseness=1] (a);
\draw[->] (c) to[bend right,looseness=1] (b);
\draw[->] (c) to[bend right,looseness=1] (a);
\end{tikzpicture}
\end{equation}
in which every possible ``backwards facing'' arrow is turned on.
\end{remark}

\section{Type $(1,\dots,1)$ cyclic quivers in arbitrary genus}\label{1...1g}
 
We begin by considering cyclic quivers whose labellings have $r_i=1$ for all $i$.  We will also from now on restrict ourselves to the case with $r$ and $d$ coprime to discount the possibility of representations which are semistable but not stable.  The quivers which we consider look like 
\begin{equation}\nonumber
\begin{tikzpicture}
\node (a) at (-0.3,0){$\bullet_{1,d_1}$};
\node (b) at (1.5,0){$\bullet_{1,d_2}$};
\node(c) at (3.2,0){$\,\,\cdots\,\,$};
\node(d) at (4.9,0){$\bullet_{1,d_n}$};

\draw[->] (a) to (b);
\draw[->] (b) to (c);
\draw[->] (c) to (d);
\draw[->] (d) to[bend right,looseness=1] (a);
\end{tikzpicture}
\end{equation}
with representations
\begin{equation}\nonumber\label{cyclicrep}
\begin{tikzpicture}
\node (a) at (0,0){$U_1$};
\node (b) at (1.5,0){$U_2$};
\node(c) at (3.2,0){$\,\,\cdots\,\,$};
\node(d) at (4.9,0){$U_n$};

\node (e) at (2.5,1.1) {$\phi_n$};

\draw[->] (a) -- (b) node[pos=0.5,below]{$\phi_1$};
\draw[->] (b) -- (c) node[pos=0.5,below]{$\phi_2$};
\draw[->] (c) -- (d) node[pos=0.5,below]{$\phi_{n-1}$};
\draw[->] (d) to [bend right,looseness=1] (a);
\end{tikzpicture}
\end{equation}
where $U_i$ is a line bundle of degree $d_i$ and $\phi_i:U_i\to U_{i+1}\otimes L$. 

Let us first consider how the corresponding automorphism group acts on a representation.  By our earlier definition of equivalence and the structure of a cyclic quiver, we have, for $\psi_i\in\mbox{Aut}(U_i)\cong \CC^\times$,
\begin{equation}\begin{split}\nonumber
\Psi\Phi\Psi^{-1} &= \left(\begin{matrix}
 0&\cdots&&\psi_1\psi_n^{-1}\phi_n\\
\psi_2\psi_1^{-1}\phi_1&\ddots&& \\
&\ddots&&\\
0&&\psi_n\psi_{n-1}^{-1}\phi_{n-1}&0\\
\end{matrix}\right).
\end{split}\end{equation}
By writing
 \begin{equation}\begin{split}\nonumber
\lambda_1&=\psi_2\psi_1^{-1}
\\
&\vdots
\\
\lambda_{n-1}&=\psi_n\psi_{n-1}^{-1}
\end{split}\end{equation}
we can realize the action of $\Psi\in\bigoplus_{i=1}^n\mbox{Aut}(U_i)$ as the action of $\left(\CC^\times\right)^{n-1}$ on the $\Phi$ part of $\mbox{Rep}(Q)$, given by
\begin{equation}\begin{split}\label{cyclicaction}
(\lambda_1,\dots,\lambda_{n-1})\cdot \Phi
=
\left(\begin{matrix}
 0&\cdots&&(\lambda_1^{-1}\dots\lambda_{n-1}^{-1})\phi_n\\
\lambda_1\phi_1&\ddots&& \\
&\ddots&&\\
0&&\lambda_{n-1}\phi_{n-1}&0\\
\end{matrix}\right)
\end{split}\end{equation}

Before moving on, we should say something about stability of cyclic quiver representations.
\begin{prop}\label{cyclicstab}
If $(U_1,\dots,U_n;\phi_1,\dots,\phi_n)$ is a representation of a type $(1,\dots,1)$ cyclic quiver in $\mbox{Bun}(X,L)$, then exactly one of the maps $\phi_i$ is allowed the possibility of being identically zero.
\end{prop}

\begin{proof}
Thoughout the proof, let the indices be counted modulo $n$.

If $\phi_i$ and $\phi_j$ with $i>j$ were both allowed to be zero, then $\bigoplus_{i=1}^n U_i$ could be presented as a direct sum of two $\Phi$-invariant subbundles, both of which have slope less than $\mu_{tot}$. This is a contradiction, meaning that at most one map can be the zero map.

To show that such a map always exists, suppose that stability imposes $\phi_i\neq 0$ for all $i=1,\dots,n$.  That is, for each $\phi_i$ there is at least one subbundle of $\bigoplus_{i=1}^n U_i$ which has slope greater than $\mu_{tot}$ and which is $\Phi$-invariant if and only if $\phi_i=0$.  For $\phi_i$, such an associated destabilizing subbundle has the form $U_j\oplus U_{j+1}\oplus\dots\oplus U_i$ for some $j$.   Now for each $i=1,\dots,n$, define $V_i$ to be the subbundle of $\bigoplus_{i=1}^nU_i$ which has these properties and has the lowest rank:
\begin{equation}\nonumber
V_i=U_{v(i)}\oplus U_{v(i)+1}\oplus\dots\oplus U_i,
\end{equation}
where $v:\{1,\dots,n\}\to\{1,\dots,n\}$.  If each $V_i$ is a line bundle, then $V_i=U_i$, and since $\mu(V_i)> \mu_{tot}$ for all $i=1,\dots,n$ we have
\begin{equation}\nonumber
\sum_{i=1}^n \mbox{deg}(U_i) = \sum_{i=1}^n \mu(V_i) > n\mu_{tot} = \sum_{i=1}^n \mbox{deg}(U_i),
\end{equation}
which is a contradiction. Therefore, we assume that at least one $V_i$ has rank greater than $1$.  For any such $V_i$ we have by definition that $\mu(U_k\oplus U_{k+1}\oplus\dots U_i) <\mu_{tot}$ for all $k$ such that $v(i)<k\leq i$.  This also tells us that $\mu(U_{v(i)}\oplus U_{v(i)+1}\oplus \dots\oplus U_{k-1}) >\mu_{tot}$.  The existence of these subbundles with slope greater than $\mu_{tot}$ gives us information about $V_k$, namely that $V_k\subset V_i$ for $v(i)<k<i$, and so in fact if $V_j \cap V_i \neq \varnothing$ for any $i\neq j$, one must be contained in the other.

Since all the $V_i$ are proper subbundles of $\bigoplus_{i=1}^nU_i$, there must exist a subset $I\subset\{1,\dots,n\}$ with $|I|>1$  such that 
\begin{equation}\nonumber
\bigcap_I V_i = \varnothing \,\,\mbox{\,\,and\,\,}\,\, \bigcup_I V_i=\bigoplus_{i=1}^n U_i.
\end{equation}
Recalling that $\mu(V_i)>\mu_{tot}$ for all $i=1,\dots,n$, we have a contradiction.  Thus, there is exactly one map $\phi_i$ which is allowed to be the zero map.
\end{proof}

With this result in our pocket, from now on we will consider all representations to have been re-indexed so that $\phi_n$ is the map which is allowed to be zero. This also provides a restriction regarding which labelled cyclic quivers we should be considering:  $Q$ admits stable representations if and only if $t\geq d_i-d_{i+1}$ for all $i=1,\dots,n-1$.  Moreover, a $(1,\dots,1)$ cyclic quiver that admits stable representations has a unique underlying $A$-type quiver that admits stable representations, which we will denote by $Q^A$. Before we can get a handle on the cyclic quiver variety, we need to understand the associated $A$-type quiver variety. The following result appears (for $\omega_X$-twisted Higgs bundles) implicitly in \cite{NJH:86} for rank $2$ and in \cite{PBG:94} for rank $3$.

\begin{lemm}\label{1...1}
Let $X$ be a curve of genus $g$, $L$ a holomorphic line bundle of degree $t$ on $X$, and $Q$ be a type $(1,\dots,1)$ cyclic quiver. Then the $\emph{SL}(n,\CC)$ and $\emph{PGL}(n,\CC)$ moduli spaces of representations of its associated $A$-type quiver $Q^A$ in $\text{Bun}(X,L)$ are
\begin{equation}\nonumber
\mathcal{M}_{X,L}(Q^A,P) \cong \tensor[^\sim]{\left(\prod_{i=1}^{n-1}\mbox{Sym}^{d_{i+1}-d_i+t}(X)\right)}{}
\end{equation}
and 
\begin{equation}\nonumber
\mathcal{M}_{X,L}^\Gamma(Q^A,P) \cong \prod_{i=1}^{n-1}\mbox{Sym}^{d_{i+1}-d_i+t}(X)
\end{equation}
respectively, where the superscript $\sim$ is used to denote an $n^{2g}$-sheeted covering.
\end{lemm}

\begin{proof}
First considering $\mbox{SL}(n,\CC)$-Higgs bundles, we are asking that a representation $(U_1,\dots,U_n;\phi_1,\dots,\phi_{n-1})$ of $Q^A$ satisifies $\text{det}\left(\bigoplus_{i=1}^n U_i\right) = P$ for some fixed $P\in\text{Jac}^d(X)$.  Since all the $U_i$ are line bundles, we have $\bigotimes_{i=1}^nU_i = P$ which tells us that one of the $U_i$ depends on the others; say 
\begin{equation}\nonumber
U_n = U_1^*U_2^*\ldots U_{n-1}^*P.
\end{equation}
Recall that $\phi_i\in H^0(X,U_i^*U_{i+1}L)\setminus \{0\}$, so $\text{deg}(\phi_i) = d_{i+1}-d_i+t$. Since we are modding out by the action of $\mathbb{C}^\times$ on this space, the information here amounts to a choice of $U_i^*U_{i+1}L$ and of projective class of $\phi_i$, which we will denote by $[\phi_i]$.  By the divisor correspondence, the information $(U_i^*U_{i+1}L,[\phi_i])$ is a point in the symmetric product of $X$ with itself $d_{i+1}-d_i+t$ times.  That is,
\begin{equation}\nonumber
\left(\left(U_1^*U_2L,[\phi_1]\right),\ldots ,\left(U_{n-1}^*U_nL,[\phi_{n-1}]\right)\right)\in\prod_{i=1}^{n-1}\mbox{Sym}^{d_{i+1}-d_i+t}(X).
\end{equation}
This is not the moduli space we are seeking,  since we want the $U_i$, not these $U_i^*U_{i+1}L$. We can, however, recover the $U_i$ by first writing $V_i = U_i^*U_{i+1}L$ and noting
\begin{equation}\begin{split}\nonumber
V_1V_2^2&V_3^3\ldots V_{n-2}^{n-2}V_{n-1}^*P\left(L^*\right)^{-1+\sum_{i=1}^{n-2}i}=U_{n-1}^n.
\end{split}\end{equation}
Thus, a point in $\prod_{i=1}^{n-1}\mbox{Sym}^{d_{i+1}-d_i+t}(X)$ fixes the $n$-th power of $U_{n-1}$.  Accounting for torsion in the Jacobian, we know that $U_{n-1}^n$ has $n^{2g}$ distinct roots \cite{A:71}.  Choosing one of these roots fixes all the other $U_i$, so $\mathcal{M}_{X,L}(Q^A,P)$ is an $n^{2g}$-fold covering of $\prod_{i=1}^{n-1}\mbox{Sym}^{d_{i+1}-d_i+t}(X)$ which we denote by $\tensor[^\sim]{\left(\prod_{i=1}^{n-1}\mbox{Sym}^{d_{i+1}-d_i+t}(X)\right)}{}$.

The $\mbox{PGL}(n,\CC)$ moduli space is realized by letting the $n$-th roots of unity $\text{Jac}^0(X)[n]$ act on $\tensor[^\sim]{\left(\prod_{i=1}^{n-1}\mbox{Sym}^{d_{i+1}-d_i+t}(X)\right)}{}$ in the following way:
\begin{equation}\nonumber
J\cdot \left(U_1,\ldots,U_n,[\phi_1],\ldots,[\phi_{n-1}]\right) = \left(J\otimes U_1,\ldots,J\otimes U_n,[\phi_1],\ldots,[\phi_{n-1}]\right).
\end{equation}
The orbits of this action consist of points $\left(U_1,\ldots,U_n,[\phi_1],\ldots,[\phi_{n-1}]\right)$ arising from different choices of the root of $U_{n-1}^n$, since if $R$ is an $n$-th root of $U_{n-1}^n$ then the others arise by tensoring $R$ with the $n$-th roots of unity in $\text{Jac}^0(X)$, and so we have our result.

\end{proof}

Now we can turn to the cyclic case and how $\calM_{X,L}(Q,P)$ and $\calM_{X,L}^\Gamma(Q^A,P)$ lie in the Higgs bundle moduli space.

\begin{lemm}\label{hitchincyclic}
The Hitchin map $h$ maps $\calM_{X,L}(Q,P)$ and $\calM_{X,L}^\Gamma(Q^A,P)$ surjectively onto $H^0(X,L^{\otimes n})\subset \calB_n$.
\end{lemm}

\begin{proof}
By definition, $h((E,\Phi)) = \mbox{char}_\lambda(\Phi)=\pm\lambda^r\pm\phi_1\dots\phi_n$. That is, the Hitchin map sends any cyclic quiver representation to the determinant of $\Phi$.  Moreover, all such determinants can be obtained.

\end{proof}

It is difficult to give a global geometric description of either of our moduli spaces, but we can exploit the Hitchin map to describe their structure in each fibre.

\begin{theorem}\label{cyclic1...1}
Given a Riemann surface $X$ of genus $g$, a holomorphic line bundle $L$ of degree $t$, and a $(1,\dots,1)$ cyclic quiver $Q$, we have the following description of the $\emph{SL}(n,\CC)$ and $\emph{PGL}(n,\CC)$ moduli spaces, parametrized by $\gamma\in H^0(X,L^{\otimes n})\subset\calB_n$:
\begin{equation}\begin{split}\nonumber
&\calM_{X,L}(Q,P)\Big|_{h^{-1}(\gamma)}
\\
&\qquad \cong \left\{(U_1,\dots,U_{n};[\phi_1],\dots,[\phi_{n-1}])\in\calM_{X,L}(Q^A,P): (\phi_1\dots\phi_{n-1})\subseteq(\gamma)\right\}
\end{split}\end{equation}
and
\begin{equation}\begin{split}\nonumber
&\calM^\Gamma_{X,L}(Q,P)\Big|_{h^{-1}(\gamma)}
\\
&\qquad \cong \left\{(U_1,\dots,U_{n};[\phi_1],\dots,[\phi_{n-1}])\in\calM_{X,L}^\Gamma(Q^A,P): (\phi_1\dots\phi_{n-1})\subseteq(\gamma)\right\}
\end{split}\end{equation}
where $Q^A$ is the associated $A$-type quiver which admits stable representations and $(\phi_1\dots\phi_{n-1})$ and $(\gamma)$ are the divisors defined by the holomorphic sections $\phi_1\dots\phi_{n-1}$ and $\gamma$ respectively.
\end{theorem}

\begin{proof}
Let us begin by only omitting $[\phi_n]$ and considering only the list$$(U_1,\dots,U_{n-1};[\phi_1],\dots,[\phi_{n-1}]).$$ Then, Lemma \ref{1...1} tells us that the moduli of such data is the $n^{2g}$-sheeted cover $\tensor[^\sim]{\left(\prod_{i=1}^{n-1}\mbox{Sym}^{d_{i+1}-d_i+t}(X)\right)}{}$, but we must also account for the fixed determinant $\phi_1\dots\phi_n=\gamma$.  Now, not all points in this cover allow for a corresponding $\phi_n\in H^0(X,U_n^*U_1L)$ to be chosen so that this condition is satisfied.  We must require that the corresponding divisors satisfy $ (\phi_1\dots\phi_{n-1})\subseteq(\gamma)$, which tells us that $(\phi_{n-1}\dots\phi_1)^{-1}\gamma$ is well-defined and holomorphic and that there is, in fact, a suitable $\phi_n$ (technically, a suitable projective class $[\phi_n]$).  We now have the above set-theoretic description of $\calM_{X,L}(Q,P)\big|_{h^{-1}(\gamma)}$, and the action of $\text{Jac}^0(X)[n]$ gives us $\calM_{X,L}^\Gamma(Q,P)\big|_{h^{-1}(\gamma)}$, just as in Lemma \ref{1...1}.  We note that these descriptions are well-defined since all $\phi_i$ in the projective class $[\phi_i]$ define the same divisor.

\end{proof}

At $\gamma =0\in H^0(X,L^{\otimes n})$ this agrees exactly with Lemma \ref{1...1}, since any divisor $(\phi_1,\dots\phi_{n-1})$ lies inside the ``divisor'' determined by $\gamma=0$.

\section{Type $(1,\dots,1)$ cyclic quivers on $\PP^1$}\label{1...1p1}

Working in $g=0$, investigating the moduli space of representations of a $(1,\dots,1)$ quiver $Q$ is simplified by the fact that $\mbox{Jac}(\PP^1)\cong\ZZ$.   In this context, the moduli are now concentrated in the maps $\phi_i$.  Using Proposition \ref{cyclicstab}, we can write
\begin{equation}\begin{split}\nonumber
\mbox{Rep}(Q) &\cong \prod_{i=1}^{n-1}\Big(H^0(\PP^1,\calO(d_{i+1}-d_i+t))\setminus\{0\}\Big)\times H^0(\PP^1,\calO(d_1-d_n+t))
\\
&\cong \prod_{i=1}^{n-1}\Big(\CC^{d_{i+1}-d_i+t+1}\setminus\{0\}\Big)\times \CC^{d_1-d_n+t+1}
\end{split}\end{equation}

as well as

\begin{equation}\nonumber
\calM_{\PP^1,\calO(t)}(Q)\cong \frac{ \prod_{i=1}^{n-1}\Big(\CC^{d_{i+1}-d_i+t+1}\setminus\{0\}\Big)\times \CC^{d_1-d_n+t+1} }{(\CC^\times)^{n-1}},
\end{equation}
where the action of $(\CC^\times)^{n-1}$ is given by Equation (\ref{cyclicaction}).  This is an interesting quotient which is reminiscent of weighted projective space.
\begin{definition}
Let $\mathbf{a}=(a_0,\dots,a_n)$, $a_i\in\NN$ and define the action of $\CC_{\mathbf{a}}^*$ on $\CC^{n+1}\setminus\{0\}$ as the following action of $\CC^\times$: $$\lambda\cdot(x_0,\dots,x_n) = (\lambda^{a_0}x_0,\dots\lambda^{a_n}x_n).$$  Then \emph{$\mathbf{a}$-weighted complex projective space} is defined as $$\PP(a_0,\dots,a_n)=\frac{\CC^{n+1}\setminus\{0\}}{\CC_{\mathbf{a}}^*}.$$ 
\end{definition}
For example, if we consider representations of the quiver
 \begin{equation}\nonumber
\begin{tikzpicture}
\node (a) at (0,0){$\bullet_{1,d_1}$};
\node (b) at (2,0){$\bullet_{1,d_2}$};
\node(c) at (-1,0){$Q=$};

\draw[->] (a) to [bend right](b);
\draw[->] (b) to[bend right] (a);
\end{tikzpicture}
\end{equation}
then the moduli space is
\begin{equation}\nonumber
\calM_{\PP^1,\calO(t)}(Q)\cong \frac{ (\CC^{d_{2}-d_1+t+1}\setminus\{0\})\times \CC^{d_2-d_1+t+1} }{\CC^\times},
\end{equation}
which is a singular analogue of weighted projective space in which we allow for negative weights (the action looks like $\lambda\cdot(\phi_1,\phi_2) = (\lambda\phi_1,\lambda^{-1}\phi_2)$).   Higher rank examples can be thought of as products of these spaces, which are somehow intertwined at the part which is acted on by negative weight. 

\begin{theorem}\label{thmcyclicp1}
The moduli space of representations of a $(1,\dots,1)$ cyclic quiver $Q$ in the category $\text{Bun}(\PP^1,\calO(t))$ is both
\begin{itemize}
\item an $\eta(Q)=\binom{nt}{d_2-d_1+t,\dots,d_n-d_{n-1}+t, d_1-d_n+t}$-sheeted covering of $H^0(\PP^1,\calO(nt))\setminus\{0\}$ which branches over points with roots of multiplicity greater than one and whose sheets intersect over $0\in H^0(\PP^1,\calO(nt))$ as $\prod_{i=1}^{n-1}\PP^{d_{i+1}-d_i+t}$; and
\item the total space of $\calO_{\prod_{i=1}^{n-1}\PP^{d_{i+1}-d_i+t}}(-1,\dots,-1)^{\oplus d_1-d_n+t+1}$, where $$\calO_{\prod_{i=1}^{n-1}\PP^{d_{i+1}-d_i+t}}(-1,\dots,-1) = \bigotimes_{i=1}^{n-1}p^*_i\calO_{\PP^{d_{i+1}-d_i+t}}(-1)$$and $p_i$ is the natural projection onto the $i$-th factor.
\end{itemize}

\end{theorem}

The first result also follows from Theorem \ref{cyclic1...1} since over $\PP^1$ the bundles $U_i$ are fixed, $\mbox{Sym}^d(\PP^1)\cong \PP^d$, and $\gamma$ and the $\phi_i$ can be thought of as polynomials, meaning that the divisor condition corresponds to the below discussion of the distribution of zeroes.

\begin{proof}
 To view the moduli space as a covering, choose a generic $\gamma \in H^0(\PP^1,\calO(nt))$ and consider the restriction $\calM_{\PP^1,\calO(t)}(Q)\big|_{h^{-1}(\gamma)}$. Fixing the determinant of $\Phi$ amounts to fixing the zeroes of the polynomial $\phi_1\dots\phi_n$. Recalling the action of $(\CC^\times)^{n-1}$ from Equation (\ref{cyclicaction}), we see that it only acts by scaling.  That is, the roots of $\phi_i$ are fixed for all $i$.  Thus, different distributions of the zeroes of $\gamma$ into the $\phi_i$ lead to legitimately different points in the moduli space.  This tells us that $\calM_{\PP^1,\calO(t)}(Q)\big|_{h^{-1}(\gamma)}$ consists of finitely many points, the number of which is given by the multinomial coefficient 
\begin{equation}\nonumber
\eta(Q):=\binom{nt}{d_2-d_1+t,\dots,d_n-d_{n-1}+t, d_1-d_n+t}.
\end{equation}

This fails over points of $H^0(\PP^1,\calO(nt))$ which, interpreted as polynomials, have repeated zeroes.  So, we have that $\calM_{\PP^1,\calO(t)}(Q)$ is an $\eta(Q)$-sheeted covering of $H^0(\PP^1,\calO(nt))$ which degenerates over points which have zeroes with multiplicity greater than one.  However, this is not quite a full description; we have so far neglected to mention the fibre $h^{-1}(0)$. Here we must always have $\phi_n=0$ and so by Lemma \ref{1...1}, $$\calM_{\PP^1,\calO(t)}(Q)\big|_{h^{-1}(0)} \cong \calM_{\PP^1,\calO(t)}(Q^A)\cong \prod_{i=1}^{n-1}\PP^{d_{i+1}-d_i+t}.$$ 

 How this fits into the covering described above can be seen by looking at certain flows of the moduli space.  On one hand, every point in $\prod_{i=1}^{n-1}\PP^{d_{i+1}-d_i+t}$ is in the intersection with the covering.  Fix such a point $p$, which we know consists only of the information
$$ \left(\begin{matrix}
 0&\cdots&&0\\
\phi_1&\ddots&& \\
&\ddots&&\\
0&&\phi_{n-1}&0\\
\end{matrix}\right).$$
 Choose any map $c\phi_n:\calO(d_n)\to\calO(d_1)\otimes\calO(t)$ and we see that the ray given by 
$$ \left(\begin{matrix}
 0&\cdots&&c\phi_n\\
\phi_1&\ddots&& \\
&\ddots&&\\
0&&\phi_{n-1}&0\\
\end{matrix}\right)$$
goes to $p$ as $c\to 0$.  That is, the cover intersects $h^{-1}(0)$ as the above product of projective spaces.

On the other hand, this behaviour describes a projection 
\begin{equation}\begin{split}\nonumber
\pi:&\mbox{ }\calM_{\PP^1,\calO(t)}(Q)\longrightarrow\calM_{\PP^1,\calO(t)}(Q^A)
\\
&(\phi_1,\dots,\phi_{n-1},c\phi_n)\mapsto (\phi_1,\dots,\phi_{n-1},0)
\end{split}\end{equation}
whose fibres are $H^0(\PP^1,\calO(d_1-d_n+t))\cong \CC^{d_1-d_n+t+1}$. We can show that this a vector bundle by looking back at the action of $(\CC^\times)^{n-1}$ on $\text{Rep}(Q)$, given by $$(\lambda_1,\dots,\lambda_{n-1})\cdot (\phi_1,\phi_2,\dots\phi_n) =  (\lambda_1\phi_1,\lambda_2\phi_2,\dots,(\lambda_1^{-1}\dots\lambda_{n-1}^{-1})\phi_n).$$
We perform reduction in stages.  That is, we first consider acting only by $\lambda_1\in\CC^\times$ on $(\phi_1,\phi_n)\in\CC^{d_2-d_1+t+1}\setminus \{0\}\times\CC^{d_1-d_n+t+1}$.  The quotient of this action is the vector bundle $\calO_{\PP^{d_2-d_1+t}}(-1)^{\oplus  d_1-d_n+t+1}$ (cf. \cite{RosTho11,Ray14}).  We note that this action commutes with direct sums: a simliar quotient of $\CC^{d_2-d_1+t+1}\setminus \{0\}\times\CC$ yields $\calO_{\PP^{d_2-d_1+t}}(-1)$.

Next we incorporate the data of $\phi_2$ and so our object of interest is $$\calO_{\PP^{d_2-d_1+t}}(-1)^{\oplus  d_1-d_n+t+1}\times\CC^{d_3-d_2+t+1}$$ being acted on by the other $\CC^\times$.  The action on the fibres of the bundle and on $\CC^{d_3-d_2+t+1}$ can be described as above, producing $\calO_{\PP^{d_3-d_2+t}}(-1)^{\oplus  d_1-d_n+t+1}$, while the Cartesian product between the base $\PP^{d_2-d_1+t}$ and the resulting projectivization of $\CC^{d_3-d_2+t+1}$ is left untwisted.  The result is a rank $d_1-d_n+t+1$ vector bundle on $\PP^{d_2-d_1+t}\times\PP^{d_3-d_2+t}$, which has degree $-1$ coming from each action.  An alternative way to describe the result of these two steps is to consider $\PP^{d_2-d_1+t}\times\PP^{d_3-d_2+t}$ with the projections $p_i$ to the two summands, and then take the pullbacks $p_1^*\calO(-1)$ and $p_2^*\calO(-1)$ of the respective tautological line bundles.  The line bundle$$\calO_{\PP^{d_2-d_1+t}\times\PP^{d_3-d_2+t+1}}(-1,-1):=p_1^*\calO(-1)\otimes p_2^*\calO(-1)$$is said to have a \emph{bi-degree}, coming from each factor.  Because$$\mbox{Pic}(\PP^{d_2-d_1+t}\times\PP^{d_3-d_2+t+1})=\ZZ\times\ZZ$$and because of the commutativity with direct sums observed above, we must have that our original double quotient is isomorphic as a vector bundle to $$\calO_{\PP^{d_2-d_1+t}\times\PP^{d_3-d_2+t+1}}(-1,-1)^{\oplus  d_1-d_n+t+1}.$$

Iterating this process for the remaining data $(\phi_3,\dots,\phi_{n-1})$ yields the desired result.

\end{proof}

The moduli space $\calM_{\PP^1,\calO(t)}(Q)$ from the point of view of the covering is pictured in Figure ~\ref{fig:covering}.
{\small{ \begin{figure}[h!]\centering
 \begin{tikzpicture}
 \draw (0,0) -- (1,1);
 \draw (1,1) -- (10,1);
 \draw (0,0) -- (9,0);
 \draw (9,0) -- (10,1);
 
   \fill[white] (5,3.5) ellipse (1.5cm and 2.75cm);
 \draw (5,3.5) ellipse (1.5cm and 2.75cm);

 \node (c) at (5,0.5) {$\bullet$};

 \node (c) at (6,0.4) {{\small{$\text{char}_\lambda\Phi=\vec{0}$}}};
 
 \node (d) at (9,0)[anchor=west]{$H^0(\PP^1,\calO(nt))$};
 
 \draw[|->] (10.25, 4.5) -- (10.25,0.5) node[pos=0.5, right]{$h$};
 
  \node (e) at (6.75,6.45) {{\small{$\mathcal{M}_{\mathbb{P}^1,\mathcal{O}(t)}(Q^A)\subset h^{-1}(0)$}}};

   \fill[white]  (5.5,4)--(6,5)--(9,5)--(8.5,4)--cycle;  
     \fill[white] (5.5,2)--(6,3)--(9,3)--(8.5,2)--cycle; 
      \fill[white]  (4.5,3)--(0.5,3) -- (1,4)--(4,4)--cycle; 
    \fill[white] (0.5,1.5) -- (1,2.5)--(4,2.5)--(4.5,1.5)--cycle;
 
 \draw [dashed] (5.5,4) -- (6,5); 
  \draw [dashed] (5.5,2) -- (6,3);  
   \draw [dashed] (4,4)--(4.5,3);   
  \draw [dashed] (4,2.5)--(4.5,1.5);

   \draw  (6,5)--(9,5)--(8.5,4)--(5.5,4);
  \draw (6,3)--(9,3)--(8.5,2)--(5.5,2); 
   \draw  (4.5,3)--(1.2,3) -- (1.7,4)--(4,4);
  \draw  (4.5,1.5)--(1.2,1.5) -- (1.7,2.5)--(4,2.5);
  
  \node (f) at (0,2.75){$\eta(Q)\text{ sheets} \Bigg\{$};
 \end{tikzpicture}
 \caption{$\calM_{\PP^1,\calO(t)}(Q)$ as a cover of $H^0(\PP^1,\calO(nt))$.}
\label{fig:covering}
 \end{figure}}}

 Recall that in the moduli space of ordinary Higgs bundles $\calM_{X,\omega_X}(r,d)$ on a curve $X$ of genus $g\geq 2$, one of the components of the $\mathbb{C}^\times$-fixed-point set is the $A$-type quiver variety $\calM_{X,\omega_X}(\bullet_{r,d})$.  This is nothing more than the moduli space of stable bundles, which has a distinguished bundle over it (namely the cotangent bundle) which is also embedded in $\calM_{X,\omega_X}(r,d)$.  Over $\PP^1$, the moduli space of stable bundles of rank at least $2$ is empty, and so this feature is absent. The bundle we construct above is an analogue of this embedded bundle for $\mathbb P^1$ and is necessarily supported on a different (nonempty) component of the fixed-point locus.  We expect this fibration structure to persist for cyclic quiver variety on curves of any genus, although with a less explicit description.

\begin{example}\label{examplecyclic}
Let $X=\PP^1$, $L=\calO(4)$, and
 \begin{equation}\nonumber
\begin{tikzpicture}
\node (a) at (0,0){$\bullet_{1,0}$};
\node (b) at (2,0){$\bullet_{1,-1}$};
\node(c) at (-1,0){$Q=$};

\draw[->] (a) to [bend right](b);
\draw[->] (b) to[bend right] (a);
\end{tikzpicture}
\end{equation}
so a representation looks like
\begin{equation}\nonumber
\begin{tikzpicture}
\node (a) at (0,0){$\calO$};
\node (b) at (2,0){$\calO(-1)$};

\node (e) at (1,0.7) {$\phi_2$};
\node (e) at (1,-0.7) {$\phi_1$};
\draw[->] (a) to [bend right](b);
\draw[->] (b) to[bend right] (a);
\end{tikzpicture}
\end{equation}
with $\phi_1\in H^0(\PP^1,\calO\otimes\calO(-1)\otimes\calO(4))\setminus \{0\}\cong \CC^4\setminus \{0\}$ and $\phi_2\in H^0(\PP^1,\calO(1)\otimes\calO\otimes\calO(4))\cong \CC^6$. Fix a generic point $\gamma\in H^0(\PP^1,\calO(8))$, say $\gamma = c(z-z_1)\dots(z-z_8)$. There are $\binom{8}{3} = 56$ ways to distribute the roots $z_i$, and using the power of the automorphism group we can put the constant $c$ with $\phi_2$:
\begin{equation}\nonumber
\Phi =  \left(\begin{matrix}
 0&c\prod_{j\in J}(z-z_j)\\
\prod_{i\in I}(z-z_i)&0 \\
\end{matrix}\right)\,\,\mbox{where }I\cap J =\varnothing, |I|=3, |J|=5.
\end{equation}
This gives a $56$-fold ramified covering of $H^0(\PP^1,\calO(8))\setminus\{0\}$.  At $\gamma=0$, we must have $\phi_2=0$, and so $$\calM_{\PP^1,\calO(4)}(Q^A)\cong \PP^3.$$  We can reach any point in $\calM_{\PP^1,\calO(4)}(Q^A)$ by choosing a suitable point in the cover and then taking $c\to 0$.  From a different point of view, the moduli space is the total space of the vector bundle $\calO_{\PP^3}(-1)^{\oplus 6}$.

\end{example}

\section{Type $(k,1)$ cyclic quivers on $\PP^1$}\label{k1p1}

We would like to expand to quivers which have some nodes labelled with higher ranks and will start by having a look at cyclic quivers of the form
 \begin{equation}\nonumber
\begin{tikzpicture}
\node (a) at (0,0){$\bullet_{k,d_1}$};
\node (b) at (2,0){$\bullet_{1,d_2}$};
\node(c) at (-1,0){$Q=$};

\draw[->] (a) to [bend right](b);
\draw[->] (b) to[bend right] (a);
\end{tikzpicture}
\end{equation}
where the underlying curve is $\PP^1$.  We will then restrict focus a little further.  Recall the splitting type of a bundle $U$ of rank $k$ over $\PP^1$
\begin{equation}\nonumber
\mathbf{a} = (a_1,\dots,a_m;s_1,\dots,s_m)
\end{equation}
which defines that $U$ splits as
 \begin{equation}U\cong \mathcal{O}(a_1)^{\plus s_1}\oplus\cdots\oplus\mathcal{O}(a_m)^{\plus s_m}.
 \nonumber\end{equation}
 
 We content ourselves with the case that $s_i=1$ for all $i=1,\dots,m$, meaning that $m=k$ and the line bundles are of mutually distinct degrees.  We further ask that\footnote{This also covers the case $\mu_{tot}>a_i$  for all $i$, simply by considering the dual quiver representation, which will have $\mu_{tot}<a_i$ for all $i$.} $\mu_{tot}<a_i$ for all $i$.  The reasons for this are discussed in Remark \ref{why}.  In effect we are imposing that an analogue of Proposition \ref{cyclicstab} holds. For the remainder of this section, we assume that $\mathbf{a}$ has these properties.

With these restrictions in place, we are considering moduli of representations which look like

\begin{equation}\label{k1cyclic}
\begin{tikzpicture}
    \node (a) at (0,0){$\mathcal{O}(a_1)$};
 \node[below=0.7cm of a] (c){$\vdots$};
 \node[below=1cm of c] (d){$\mathcal{O}(a_k)$};
 \node[right=3cm of c] (g){$\mathcal{O}(d_2)$};
\draw[->] (a) to[bend left=7] node[above] {$\phi_1$} (g);
\draw[->] (g) to[bend left=7] node[below] {$\phi_2$} (a);
\draw[->] (d) to[bend left=7] node[above] {$\phi_{2k-1}$} (g);
\draw[->] (g) to[bend left=7] node[below] {$\phi_{2k}$} (d);
    \end{tikzpicture}
\end{equation}
along with automorphisms $\psi_{ij}:\calO(a_i)\to\calO(a_j)$ for all $i>j$.  Here, stability implies that none of $\phi_1,\phi_3,\dots,\phi_{2k-1}$ can be zero, but any of $\phi_2,\phi_4,\dots,\phi_{2k}$ can be.  This imposes the further condition $-a_i+d_2+t\geq 0$ for all $i=1,\dots,k$.  We denote the moduli space of such representations by $\calM_{\PP^1,\calO(t)}(Q;\mathbf{a})$.  Let this notation further imply that no other splitting types are ``allowed'' to occur within this moduli space (see Remark \ref{splittingtypes}). 

\begin{lemm}
\begin{equation}\nonumber
h:\calM_{\PP^1,\calO(t)}(Q;\mathbf{a}) \twoheadrightarrow H^0(\PP^1,\calO(t)^{\otimes 2})
\end{equation}
\end{lemm}

\begin{proof}
This follows from the same argument as Lemma \ref{hitchincyclic}.  In this case the calculation gives
\begin{equation}\nonumber
h((E,\Phi)) = \pm\lambda^r\mp\lambda^{r-1}(\phi_1\phi_2+\phi_3\phi_4+\dots+\phi_{2k-1}\phi_{2k}).
\end{equation}
\end{proof}

Our strategy to understand $\calM_{\PP^1,\calO(t)}(Q;\mathbf{a})$ is to view a representation such as (\ref{k1cyclic}) as $k$ separate $(1,1)$ quiver representations by first ``using up'' the freedom of the automorphisms $\psi_{ij}$.  On the projective line, these maps have interpretations as polynomials on $\PP^1$, a fact which we have already made use of.  This allows us to perform a Euclidean reduction such as $\phi_j\mapsto \phi_j+\phi_i\psi_{ji}$.  This process is described in detail in \cite{RaySun18}. To state our result, we also need to use the following notation:

Let $Q$ be the type $(1,1)$ cyclic quiver \begin{equation}\nonumber
\begin{tikzpicture}
\node (a) at (0,0){$\bullet_{1,d_1}$};
\node (b) at (2,0){$\bullet_{1,d_2}$};
\node(c) at (-1,0){$Q=$};

\draw[->] (a) to [bend right](b);
\draw[->] (b) to[bend right] (a);
\end{tikzpicture}
\end{equation}
oriented with $d_1>d_2$ so that in a representation, the map $\phi_1:\calO (d_1)\to\calO (d_2)$ can not be zero by stability.  Denote by $\calM_{\PP^1,\calO(t)}(Q^{-b})$ the moduli space of representations of $Q$ in $\text{Bun}(\PP^1,\calO(t))$ where $\phi_1$ has had its amount of freedom (in terms of complex dimensions) reduced by $b$.  

\begin{example}
For example, let $L=\calO(4)$ and 
 \begin{equation}\nonumber
\begin{tikzpicture}
\node (a) at (0,0){$\bullet_{1,0}$};
\node (b) at (2,0){$\bullet_{1,-1}$};
\node(c) at (-1,0){$Q=$};

\draw[->] (a) to [bend right](b);
\draw[->] (b) to[bend right] (a);
\end{tikzpicture}
\end{equation}
as in Example \ref{examplecyclic}.  Consider $\calM_{\PP^1,\calO(t)}(Q^{-2})$, saying that now $\phi_1\in \CC^2\setminus \{0\}$ and $\phi_2\in \CC^6$.  Now for generic $\gamma\in H^0(\PP^1,\calO(8))$ reduced to $H^0(\PP^1,\calO(6))$, we have only to distribute $6$ zeroes into the $\phi_1$ and $\phi_2$. We have $\binom{6}{1} =6$ ways to do so.  At $\gamma=0$, we have $\calM_{\PP^1,\calO(t)}(Q^{-2})\big|_{h{-1}(0)} \cong \PP^1$.
\end{example}

Note that even though we may not have an interpretation of the information of $\calM_{\PP^1,\calO(t)}(Q^{-2})$ as an unadjusted $(1,1)$ cyclic quiver variety (although we do in some cases), its structure is easily calculated in a familiar way.  Now we can write the moduli space of representations of a $(k,1)$ cyclic quiver in terms of these adjusted moduli spaces of representations of $(1,1)$ cyclic quivers.

\begin{theorem}\label{k1cyclicp1}
Let $Q$ be a type $(k,1)$ cyclic quiver, $\mathbf{a}=(a_1,\dots,a_k;1,\dots,1)$ be a splitting type, and $Q_i$ be the quivers
 \begin{equation}\nonumber
\begin{tikzpicture}
\node (a) at (0,0){$\bullet_{1,a_i}$};
\node (b) at (2,0){$\bullet_{1,d_2}$};

\draw[->] (a) to [bend right](b);
\draw[->] (b) to[bend right] (a);
\end{tikzpicture}
\end{equation} 
Then
\begin{equation}\nonumber
\mathcal{M}_{\PP^1,\calO(t)}(Q;\mathbf{a})\cong\calM_{\PP^1,\calO(t)}(Q_1)\times\prod_{i=2}^k\calM_{\PP^1,\calO(t)}\Big(Q^{-\sum_{j=1}^{i-1}(a_j-a_i+1)}_i\Big).
\end{equation} 
\end{theorem}

\begin{proof}
Recall the visualization of a representation from Equation \ref{k1cyclic}. Due to our assumptions on $\mathbf{a}$, the maps $\phi_1,\phi_3,\dots,\phi_{2k-1}$ cannot be zero, but any of $\phi_2,\phi_4,\dots,\phi_{2k}$ can be.  The automorphisms $\psi_{ij}$ are not able reduce the amount of freedom of any of the maps $\phi_2,\phi_4,\dots,\phi_{2k}$.  This means that the reductions which take place are exactly the ones that would take place in the $A$-type $(k,1)$ case with $\mathbf{a}$ as the splitting type.  In particular, the map $\phi_{2i-1}$ has its moduli reduced by $\sum_{j=1}^{m-1}(a_j-a_i+1)$.  Note that $\phi_1$ is not reduced at all.  Now we have split our moduli problem into $k$ parts, each of the form $\calM_{\PP^1,\calO(t)}\Big(Q_i^{-\sum_{j=1}^{i-1}(a_j-a_i+1)}\Big)$, and we have our result.

\end{proof}

\begin{remark}\label{splittingtypes}
Strictly speaking, the moduli space $\mathcal{M}_{\PP^1,\calO(t)}(Q)$ is stratified by splitting types of the bundle $U_1$.  If one fixes a splitting type $\mathbf{a}$, there are other ``less generic'' splittings and the moduli spaces corresponding to these can be viewed as lying inside the moduli space corresponding to $\mathbf{a}$.  Along this locus, the automorphisms $\psi_{ij}$ have more freedom and the moduli space is blown down in some way.  Since our approach only allows us to realize the moduli space for specific splittings types $\mathbf{a}$, we will not explore this stratification.  The space which we construct above could be alternatively described as the projective completion of the regular part of the moduli space corresponding to splitting type $\mathbf{a}$.  This phenomenon is investigated for some special cases of $A$-type quivers in \cite{RaySun18}.
\end{remark}

\begin{cor}
The moduli space restricted to a fibre, $\mathcal{M}_{\PP^1,\calO(t)}(Q;\mathbf{a})\big|_{h^{-1}(\gamma)}$, is a $$\binom{(r-1)t-\sum_{j=1}^{k-1}(a_j-a_k+1)}{d_2-a_k+t-\sum_{j=1}^{k-1}(a_j-a_k+1)}$$ -to-one covering of $$\calM_{\PP^1,\calO(t)}(Q_1)\times\prod_{i=2}^{k-1}\calM_{\PP^1,\calO(t)}\Big(Q^{-\sum_{j=1}^{i-1}(a_j-a_i+1)}_i\Big),$$ except over points $(\phi_1,\phi_2,\dots,\phi_{2k-2})$ such that $\phi_1\phi_2+\phi_3\phi_4+\dots+\phi_{2k-3}\phi_{2k-2}=-\gamma$, where the cover intersects the fibre as $\PP^{d_2-a_k+t-\sum_{j=1}^{k-1}(a_j-a_k+1)}$.
\end{cor}

\begin{proof}
If we fix $\gamma = \phi_1\phi_2+\phi_3\phi_4+\dots+\phi_{2k-1}\phi_{2k}$, it is clear that we have the freedom to choose any $(\phi_1,\phi_2,\dots,\phi_{2k-2})$ which will then place restrictions on $\phi_{2k-1}$ and $\phi_{2k}$.  $(\phi_1,\phi_2)$ must be chosen before we can reduce the freedom of $\phi_3$, and so forth, which is why $(\phi_{2k-1},\phi_{2k})$ is the last to be chosen.  Now the problem amounts to distributing the zeroes of $\gamma-\phi_1\phi_2-\phi_3\phi_4-\dots-\phi_{2k-3}\phi_{2k-2}$ which are not already fixed into $\phi_{2k-1}$ and $\phi_{2k}$. In the case $\phi_1\phi_2+\phi_3\phi_4+\dots+\phi_{2k-3}\phi_{2k-2}=-\gamma$, we must have $\phi_{2k}=0$, and so the fibre is the projective space $\PP^{d_2-a_k+t-\sum_{j=1}^{k-1}(a_j-a_k+1)}$.

\end{proof}

 In contrast to the $(1,\dots,1)$ cyclic case, in the $(k,1)$ cyclic case there is nothing unusual happening to the moduli space at the nilpotent cone $h^{-1}(0)$.  Instead, there is an analogue of $\calM_{\PP^1,\calO(t)}(Q^A)$ living inside each fibre.

\begin{example}
Let $t=5$, $Q$ be the quiver
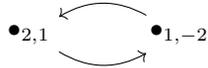
\begin{equation}\nonumber
\begin{tikzpicture}
\node (a) at (0,0){$\bullet_{2,1}$};
\node (b) at (2,0){$\bullet_{1,-2}$};

\draw[->] (a) to [bend right](b);
\draw[->] (b) to[bend right] (a);
\end{tikzpicture}
\end{equation}
and $\mathbf{a}$ be the splitting type $(1,0;1,1)$.  So a representation looks like
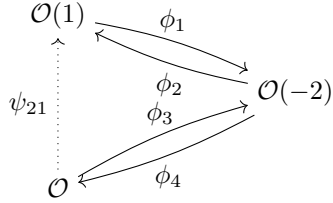
\begin{equation}\nonumber
\begin{tikzpicture}
    \node (a) at (0,0){$\mathcal{O}(1)$};
 \node[below=0.6cm of a] (c){};
 \node[below=0.9cm of c] (d){$\mathcal{O}$};
 \node[right=2.4cm of c] (g){$\mathcal{O}(-2)$};
\draw[->] (a) to[bend left=7] node[above] {$\phi_1$} (g);
\draw[->] (g) to[bend left=7] node[below] {$\phi_2$} (a);
\draw[->] (d) to[bend left=7] node[above] {$\phi_3$} (g);
\draw[->] (g) to[bend left=7] node[below] {$\phi_4$} (d);
\draw[->, dotted] (d) to node[left] {$\psi_{21}$} (a);
    \end{tikzpicture}
\end{equation}
where $\phi_1\in\CC^3\setminus\{0\}$, $\phi_2\in\CC^9$, $\phi_3\in\CC^4\setminus\{0\}$, and $\phi_1\in\CC^8$.  The automorphism $\psi_{21}\in\CC^2$ reduces $\phi_3$ to $\CC^2\setminus\{0\}$, and 
\begin{equation}\nonumber
\mathcal{M}_{\PP^1,\calO(5)}(Q;\mathbf{a})\cong\calM_{\PP^1,\calO(5)}(Q_1)\times\calM_{\PP^1,\calO(5)}(Q_2^{-2}).
\end{equation} 
Considering $\mathcal{M}_{\PP^1,\calO(5)}(Q;\mathbf{a})\big|_{h^{-1}(\gamma)}$, we fix $\gamma = c(z-z_1)\dots(z-z_{10}) = \phi_1\phi_2+\phi_3\phi_4$.  We can choose any $(\phi_1,\phi_2)\in\calM_{\PP^1,\calO(5)}(Q_1)$ and then consider $\gamma-\phi_1\phi_2=\phi_3\phi_4$.  We know that the map $\phi_3$ is reduced by $\phi_{21}$, so we ignore the top $2$ degrees of $\gamma-\phi_1\phi_2$.  The moduli problem then amounts to distributing $8$ zeroes, $1$ into $\phi_3$ and $7$ into $\phi_4$, resulting in an $8$-fold covering of $\calM_{\PP^1,\calO(5)}(Q_1)$, except over the points $\gamma=\phi_1\phi_2$, where we have $\PP^1$.

This same behaviour is displayed in the intersection with the nilpotent cone, although we can identify the locus $$\{(\phi_1,\phi_2)\in\calM_{\PP^1,\calO(5)}(Q_1) :\phi_1\phi_2=0\}\times\PP^1\cong \PP^2\times\PP^1$$ with $\calM_{\PP^1,\calO(t)}(\bullet_{2,1}\longrightarrow\bullet_{1,-2})$.
\end{example}

\begin{remark}\label{why}
The reason for the specific structure of splittings $\mathbf{a}$ we consider is that when the maps $\phi_i$ which are and are not allowed to be zero by stability are less rigidly structured, the actions of $\psi_{ij}$ become less clear.  Without being able to say exactly which maps the automorphisms reduce, our approach of considering $k$ different $(1,1)$ cyclic quiver varieties is less effective.   It is also this lack of a clear decomposition into products of varieties which we understand that prevents us from using this procedure to study splittings where some of the line bundles $\calO(a_i)$ have the same degree, as well as $(1,k,1)$ and general argyle quivers.  In these later cases, one must also contend with the fact that the terms of the fixed characteristic polynomial are, in general, no longer simply products of the maps $\phi_i$. 
\end{remark}

\bibliographystyle{acm} 
\bibliography{CyclicBibliography}

\end{document}